\documentclass{amsart}
\usepackage{datetime}
\usepackage{amssymb,amsmath,amsthm}
\usepackage{graphicx}
\usepackage{enumerate}
\usepackage{mathtools}

\usepackage{hyperref}

\usepackage{xspace}

\DeclareMathOperator{\supp}{supp}
\DeclareMathOperator{\dist}{dist}
\newcommand{\TV}{{\textup{\textsf{TV}}}}
\newcommand{\ve}{{\textup{\textsf{v}}}}

\numberwithin{equation}{section}
\theoremstyle{plain}
\newtheorem{thm}[equation]{Theorem}
\theoremstyle{definition}

\theoremstyle{plain}
\newtheorem{lem}[equation]{Lemma}

\theoremstyle{definition}
\newtheorem{rem}[equation]{Remark}
\newtheorem*{acknowledgement*}{Acknowledgement}

\catcode`\@=11
\@addtoreset{equation}{section}
\numberwithin{equation}{section}

\catcode`\@=11
\@addtoreset{table}{section}
\numberwithin{table}{section}

\catcode`\@=11
\@addtoreset{figure}{section}
\numberwithin{figure}{section}

\DeclareMathOperator*{\argmin}{argmin}

\newcommand{\DIV}{{\textrm{div}}}   
\newcommand{\GRAD}{\nabla}           

\def\Lun{{L^1(\Omega)}}

\def\Ldeux{{{  L}^2   (\Omega)}}

\def\Hsp1d{{{   \bf H}^{s+1}   (\Omega)}}

\def\BV{BV(\Omega)}
\newcommand{\Real}{\mathbb R}
\newcommand{\diff}{\, \mbox{\rm d}}
\newcommand{\vare}{{\varepsilon}}
\newcommand{\dt}{{\Delta t}}

\newcommand{\ie}{i.e.,\@\xspace}
\newcommand{\eg}{e.g.\@\xspace}
\newcommand{\cf}{cf.\@\xspace}

\newcommand{\ue}{\textup{\textsf{u}}}

\def\scl{\left\langle}
\def\scr{\right\rangle}

\newcommand{\be}{{\bf e}}

\newcommand{\calC}{{\mathcal C}}

\newcommand{\calI}{{\mathcal I}}

\newcommand{\calN}{{\mathcal N}}
\newcommand{\calO}{{\mathcal O}}

\newcommand{\calT}{{\mathcal T}}

\newcommand{\polN}{{\mathbb N}}

\newcommand{\polP}{{\mathbb P}}
\newcommand{\polQ}{{\mathbb Q}}


\begin{document}
\title[TVD Interpolation]{A Total Variation Diminishing Interpolation Operator and Applications}

\author[R.H.~Nochetto]{Ricardo H.~Nochetto}
\address[R.H.~Nochetto]{Department of Mathematics and Institute for Physical Science and Technology,
University of Maryland, College Park, MD 20742, USA.}
\email{rhn@math.umd.edu}

\author[A.J.~Salgado]{Abner J.~Salgado}
\address[A.J.~Salgado]{Department of Mathematics, University of Maryland, College Park, MD 20742, USA.}
\email{abnersg@math.umd.edu}

\thanks{
This work is partially supported by NSF grants DMS-0807811 and DMS-1109325.
AJS is also partially supported by an AMS-Simons grant.
}

\keywords{Total Variation, Interpolation, Approximation, Finite Elements}

\subjclass[2010]{65D05; 
49M25;  
65K15;  
65M60;  
65N15;  
49J40   
}

\date{Submitted to Mathematics of Computation on \today.}

\begin{abstract}
We construct, on continuous $\polQ_1$ finite elements over Cartesian meshes,
an interpolation operator that does not increase the total variation.
The operator is stable in $L^1$ and exhibits second order approximation properties.
With the help of it we provide improved error estimates
for discrete minimizers of the total variation denoising problem and for
total variation flows.
\end{abstract}

\maketitle

\section{Introduction}
\label{sec:intro}
The approximation of weakly differentiable functions by polynomials is a very useful tool 
when trying to understand the behavior of such functions. In fact, S.L.~Sobolev himself \cite{MR0052039}
used a kind of averaged Taylor polynomial to discuss equivalent norms in the celebrated spaces
that now bear his name. Polynomial approximation also plays a crucial r\^ole in numerics,
as it is the basis of the analysis of finite element methods and discretization 
techniques of partial differential equations. It is of no surprise then, that there exists several
different constructions and that the approximation properties of such polynomials are well studied.
We refer the reader to \cite{MR1742264,MR0400739,MR2047080,MR559195,MR714693,MR2164092,MR1933037,MR1011446}
for a by no means exhaustive list that refers, mostly, to such results
in the finite element method context.

The quasitotality of the aforementioned works construct an operator that is well defined in $L^1(\Omega)$
(see Section~\ref{sec:notation} for notation) and is stable in $L^p(\Omega)$ for $p\geq1$. In addition,
their approximation properties are studied when the functions belong to certain Sobolev spaces.
However, in problems dealing with free discontinuities or connected with geometric
features such as minimization of coarea; mean curvature flows and others, it is most
natural to work in the space $\BV$ and so the need for an approximation theory by piecewise polynomials
in this space arises.


In prior work \cite{TVFlow} we dealt with such issue by first regularizing the function and
then taking the Lagrange interpolant of the regularization. By properly choosing the regularization
parameter as a function of the mesh size, we were able to obtain the properties that served our purposes.
However, much simpler arguments and better results would have been possible if a 
\emph{TV diminishing interpolant} was available. In other words, 
we wonder about an operator
\begin{equation}
  \Pi_h : \Lun \rightarrow X_h : \qquad \int_\Omega |\GRAD \Pi_h w | \leq \int_\Omega |\GRAD w|,
  \quad \forall w \in \BV,
\label{eq:tvd-def}
\tag{TVD}
\end{equation}
where we denote by $X_h$ the finite element space;
note that the stability constant in \eqref{eq:tvd-def} is $1$.
The purpose of this work is to construct such
an operator under the, rather stringent, geometric assumption that the underlying mesh is Cartesian
and the finite element space is made of continuous piecewise multilinear elements.

We organize our presentation as follows. Notation, functional and discrete spaces are introduced in
Section~\ref{sec:notation}. The core of our discussion is Section~\ref{sec:interpolant}, where we
construct our TV diminishing interpolant and prove its main properties
for periodic functions. The extension to non-periodic functions is in Section~\ref{sec:nonperiodic}.
Finally, some applications are discussed in Section~\ref{sec:Applications}; namely we provide error estimates
for total variation minimization and total variation flows.

\section{Notation and preliminaries}
\label{sec:notation}

In what follows we consider either $\Omega = S^1\times S^1$ (the two dimensional torus) or
$\Omega = (0,1)^2$ (the unit square).
Our arguments can be extended without difficulty to more space dimensions.
If $z \in \Real^2$ we denote its coordinates by $z = (z^1,z^2)$.

We denote by $L^p(\Omega)$ with $p\in[1,\infty]$ the space
of Lebesgue integrable functions with exponent $p$. Sobolev spaces will be denoted by
$W^s_p(\Omega)$, where $s \in \Real$ is the differentiability order.
Whenever $X$ is a normed space, we denote by $\|\cdot\|_X$ its norm.

We say that a function $w\in \BV$ (is of bounded variation), if $w\in \Lun$ and its derivative
in the sense of distributions is a Radon measure. To simplify notation, we will use $\GRAD w$ to 
denote such measure. We define the total variation (TV) by
\begin{equation}
  \TV( w ) = \int_\Omega |\GRAD w|,
\label{eq:BVseminorm}
\end{equation}
and we endow the space $\BV$ with the norm 
\[
  \| w \|_{BV} = \| w \|_{L^1} + \TV(w).
\]
We refer the reader to \cite{MR1857292,MR1014685} for the most relevant facts about such functions.

To carry out the finite element discretization, we introduce the number of points in each direction
$N_i \in \polN$, $i=1,2$. The mesh size in each direction is
\[
  h_i = \frac1{N_i}, \quad i=1,2.
\]
The nodes of the mesh are
\[
  z_{k,l}= (k h_1, l h_2), \quad k = \overline{0,N_1}, \ l = \overline{0,N_2}.
\]
The triangulation is
\[
  \calT = \{ T_{k,l} \}_{k=0,l=0}^{N_1-1,N_2-1},
\]
where
\[
  T_{k,l} = \left\{ z \in \Real^2:\ z_{k,l}^1 < z^1 < z_{k+1,l}^1, \ z_{k,l}^2 < z^2 < z_{k,l+1}^2 \right\}.
\]
Clearly $\bar\Omega = \cup_{T \in \calT} \bar T$.
On the basis of such triangulation we define
\begin{equation}
  X_h = \left\{ w_h \in \calC^0(\bar\Omega): \ w_h |_T \in \polQ_1, \ \forall T \in \calT \right\}.
\label{eq:defofFEM}
\end{equation}
Notice that, given the special structure of the mesh, if we denote by $\Lambda_{k,l}$ the
Lagrange basis function associated with node $z_{k,l}$, then
\[
  \Lambda_{k,l}(z) = \lambda_k^1(z^1) \lambda_l^2(z^2),
\]
where $\lambda_k^i$, $i=1,2$,
denotes the one-dimensional Lagrange basis function associated with node $z_k^i$ on a 
mesh of size $h_i$.

\section{TV diminishing interpolant: periodic domains}
\label{sec:interpolant}

In this Section we construct the TV diminishing interpolation operator in the torus,
$\Omega = S^1 \times S^1$, and prove some of
its approximation properties. It is based on a weighted averaged Taylor polynomial as in
\cite{MR1742264,MR0400739,MR2164092} but, to better exploit the symmetries related to our mesh, we use a very
specific weighting function.

To begin with, let $\chi_I$ be the characteristic function of the interval $I$ and set
\[
  \psi^i = \frac1{h_i} \chi_{[-h_i/2,h_i/2]}, \quad i=1,2.
\]
Notice that $\int \psi^i = 1$ and $\| \psi^i \|_{L^\infty} = 1/h_i$. The weighting function is then
\begin{equation}
  \psi(z) = \psi^1(z^1) \psi^2(z^2).
\label{eq:weight}
\end{equation}
Notice that
\begin{equation}
  Q = \supp \psi = \left[-\frac{h_1}2, \frac{h_1}2 \right] \times \left[-\frac{h_2}2, \frac{h_2}2 \right].
\label{eq:defofQ}
\end{equation}

Given a function $w\in \Lun$, similarly to \cite{MR1742264,MR0400739,MR2164092}, we define
\begin{equation}
  W_{k,l} = \int_Q w(z_{k,l} + \zeta) \psi(\zeta) \diff \zeta \in \Real.
\label{eq:defofWkl}
\end{equation}
The interpolant $\Pi_h w \in X_h$ is the unique function that satisfies
\begin{equation}
  \Pi_h w (z_{k,l}) = W_{k,l}.
\label{eq:defofPi}
\end{equation}

The stability of this operator is as follows.

\begin{lem}[$L^p$-stability]
The operator $\Pi_h: \Lun \rightarrow X_h$ is well defined and stable in any $L^p(\Omega)$,
$1\leq p \leq \infty$, namely
\begin{equation}
  \| \Pi_h w \|_{L^p} \leq \| w \|_{L^p}, \quad \forall w \in L^p(\Omega).
\label{eq:Lpstab}
\end{equation}
\end{lem}
\begin{proof}
It suffices to show \eqref{eq:Lpstab} for $p=1$ and $p=\infty$. The intermediate powers will
follow by function space interpolation.
Estimate \eqref{eq:Lpstab} for $p=\infty$ is evident. To obtain the
$L^1$-estimate we consider $T_{\kappa,\ell} \in \calT$ and
\[
  \| \Pi_h w \|_{L^1(T_{\kappa,\ell})} = 
  \int_{T_{\kappa,\ell}} \left| 
    \sum_{k=\kappa,l=\ell}^{\kappa+1,\ell+1} W_{k,l} \Lambda_{k,l}(z)
  \right| \diff z \leq
  \sum_{k=\kappa,l=\ell}^{\kappa+1,\ell+1} |W_{k,l}| \int_{T_{\kappa,\ell}} \Lambda_{k,l}(z) \diff z,
\]
Notice that, if $k = \kappa, \kappa+1$ and $l = \ell, \ell+1$,
\[
  \int_{T_{\kappa,\ell}} \Lambda_{k,l}(z) \diff z =
  \int_{z_{\kappa,\ell}^1}^{z_{\kappa+1,\ell}^1} \lambda_k^1(z^1) \diff z^1
  \int_{z_{\kappa,\ell}^2}^{z_{\kappa,\ell+1}^2} \lambda_l^2(z^2) \diff z^2 = \frac14 h_1 h_2.
\]
By definition \eqref{eq:defofWkl},
\[
  |W_{k,l}| \leq \frac{1}{h_1 h_2} \int_{z_{k,l}+Q} |w(\zeta)| \diff \zeta.
\]
From the considerations given above it follows that
\[
  \| \Pi_h w \|_{L^1} =
  \sum_{\kappa=0,\ell=0}^{N_1-1,N_2-1} \| \Pi_h w \|_{L^1(T_{\kappa,\ell})}
  \leq \frac14 \sum_{\kappa=0,\ell=0}^{N_1-1,N_2-1} \sum_{k=\kappa,l=\ell}^{\kappa+1,\ell+1}
  \int_{z_{k,l}+Q} |w(\zeta)| \diff \zeta,
\]
To conclude it suffices to notice that every vertex belongs to exactly four cells.
\end{proof}

The main result of this contribution is the following.

\begin{thm}[TV diminishing: periodic case]
\label{thm:tvd}
The operator $\Pi_h: \Lun \rightarrow X_h$ defined in \eqref{eq:defofPi} satisfies \eqref{eq:tvd-def}.
\end{thm}
\begin{proof}
Notice that, thanks to the particular structure of the mesh, the basis functions satisfy
\[
  \partial_1 \Lambda_{k,l} = \left( \lambda_k^1 \right)' \lambda_l^2 = \pm \frac1{h_1} \lambda_l^2.
\]

If $w\in \calC^1(\Omega)$ and $T=T_{k,l} \in \calT$, then
\begin{equation}
  \begin{aligned}
      \partial_1 \Pi_h w|_T &= W_{k,l} \left(\frac{-1}{h_1}\right) \lambda_l^2
      + W_{k+1,l} \left(\frac{1}{h_1}\right) \lambda_l^2 \\
      & + W_{k,l+1} \left(\frac{-1}{h_1}\right) (1-\lambda_l^2)
      + W_{k+1,l+1} \left(\frac{1}{h_1}\right) (1-\lambda_l^2) \\
      &= \frac1{h_1} \left(  W_{k+1,l} -  W_{k,l} \right) \lambda_l^2
      + \frac1{h_1} \left(  W_{k+1,l+1} -  W_{k,l+1} \right) (1-\lambda_l^2),
  \end{aligned}
\label{eq:sharp}
\end{equation}
where we have used that, for every $l$, we have $\lambda_l^2 + \lambda_{l+1}^2 \equiv 1$ on
$[z_{k,l}^2,z_{k,l+1}^2]$.
In view of the definition of $W_{k,l}$, given in \eqref{eq:defofWkl}, we can write
\begin{align*}
  W_{k+1,l} -  W_{k,l} &=
  \int_Q \left( w(z_{k+1,l} + \zeta) - w(z_{k,l} + \zeta) \right) \psi(\zeta) \diff \zeta \\
  &= h_1 \int_Q \int_0^1 \partial_1 w (s z_{k+1,l} + (1-s) z_{k,l} + \zeta) \psi(\zeta) \diff s \diff \zeta.
\end{align*}
Notice that, at this point, it was essential to have a $\calC^1$ function.
Notice also that
\[
  \int_T \lambda_l^2(z^2) \diff z = \int_T (1-\lambda_l^2(z^2)) \diff z = \frac12 |T| = \frac12 h_1 h_2.
\]
With these observations we estimate the local variation in the first coordinate direction
of the interpolant as
\begin{align*}
  \int_T |\partial_1 \Pi_h w(z)| \diff z &\leq
  \frac12 h_1 h_2 \int_Q \int_0^1
  \left| \partial_1 w (s z_{k+1,l} + (1-s) z_{k,l} + \zeta) \right| \psi(\zeta) \diff s \diff \zeta \\
  &+ \frac12 h_1 h_2 \int_Q \int_0^1
  \left| \partial_1 w (s z_{k+1,l+1} + (1-s) z_{k,l+1} + \zeta) \right| \psi(\zeta) \diff s \diff \zeta.
\end{align*}
The total variation in the first coordinate direction is then
\begin{align*}
  \int_\Omega |\partial_1 \Pi_h w | &= 
  \sum_{l=0}^{N_2-1}\sum_{k=0}^{N_1-1} \int_{T_{k,l}} |\partial_1 \Pi_h w(z)| \diff z\\ 
  &\leq 
  \frac12 h_1 h_2 \int_Q \sum_{l=0}^{N_2-1}\sum_{k=0}^{N_1-1} \int_0^1 \left(
      \left| \partial_1 w (s z_{k+1,l} + (1-s) z_{k,l} + \zeta) \right|
  \right. \\
  &\left.
    + \left| \partial_1 w (s z_{k+1,l+1} + (1-s) z_{k,l+1} + \zeta) \right|
  \right) \psi(\zeta)  \diff s \diff \zeta.
\end{align*}
Introduce the change of variables
$(0,1) \ni s \mapsto \sigma_k \in (z_{k,l}^1+\zeta^1,z_{k+1,l}^1+\zeta^1)$ given by
\[
  \sigma_k = s z_{k+1,l}^1 + (1-s)z_{k,l}^1 + \zeta^1,
\]
and notice that with it we can add over $k$ to obtain
\begin{align*}
  \int_\Omega |\partial_1 \Pi_h w | &\leq
  \frac12 h_2 \int_Q \sum_{l=0}^{N_2-1} \left(
  \int_0^1 |\partial_1 w(\sigma,\zeta^2 + lh_2) | \diff \sigma 
  \right. \\
  & \left.
  + \int_0^1 | \partial_1 w (\sigma,\zeta^2 + (l+1)h_2 )| \diff \sigma
  \right) \psi(\zeta) \diff\zeta \\
  &\leq h_2 \int_{-h_1/2}^{h_1/2} \psi^1(\zeta^1) \diff \zeta^1
  \sum_{l=0}^{N_2-1} \int_{-h_2/2}^{h_2/2} \int_0^1 | \partial_1 w (\sigma, \zeta^2 + lh_2)| 
  \psi^2(\zeta^2) \diff \sigma \diff \zeta^2,
\end{align*}
where, to arrive at the last inequality, we have used the definition of $\psi$ given in \eqref{eq:weight}.
Finally, we recall that $\| \psi^2 \|_{L^\infty} = \frac1{h_2}$, $\int \psi^1 =1$, and
\[
  \sum_{l=0}^{N_2-1} \int_{-h_2/2}^{h_2/2} \int_0^1 | \partial_1 w (\sigma, \zeta^2 + l h_2)| \diff \sigma
  \diff \zeta^2
  = \int_\Omega | \partial_1 w |,
\]
to conclude the proof for a $\calC^1$ function $w$.

In general, \ie $w\in\BV \setminus \calC^1(\Omega)$, we use an approximation argument.
We recall \cite[Theorem 5.3.3]{MR1014685} that smooth functions are dense in $\BV$ under strict
convergence. In other words,
there is a sequence $\{w_n\}_{n\in\polN} \subset \calC^\infty(\Omega)$ such that
\[
  \lim_{n\rightarrow \infty} \| w_n - w \|_{L^1} = 0,
  \qquad
  \limsup_{n \rightarrow \infty} \TV(w_n) \leq \TV(w).
\]
Since $\TV(\Pi_h w_n) \leq \TV(w_n)$ for all $n \in \polN$, we deduce
\[
  \TV(\Pi_h w ) = \lim_{n\rightarrow \infty} \TV(\Pi_h w_n) \leq
  \liminf_{n\rightarrow} \TV(w_n) \leq \TV(w),
\]
and thus conclude the proof.
\end{proof}

The approximation properties of this operator are summarized in the following.

\begin{thm}[Approximation]
\label{thm:approx}
Let $h = \max\{ h_i: \ i=1,2 \}$.
If $w \in \BV$, then
\begin{equation}
  \| w - \Pi_h w \|_{L^1} \leq c h \ \TV(w),
\label{eq:BVoptimalapprox}
\end{equation}
with $c>0$ a geometric constant.
If $w \in \BV \cap L^p(\Omega)$ with $p>1$, then for all $q\in [1,p]$,
\begin{equation}
  \| w - \Pi_h w \|_{L^q} \leq c h^{1-s} \TV(w)^{1-s} \| w \|_{L^p}^s,
  \quad \frac1q = \frac{1-s}1 + \frac{s}{p}.
\label{eq:Lqapprox}
\end{equation}
If $w\in W^2_p(\Omega)$ with $1\leq p \leq \infty$, then
\begin{equation}
    \| w - \Pi_h w \|_{L^p} \leq c h^2 |w|_{W^2_p}.
\label{eq:secondorder}
\end{equation}
\end{thm}
\begin{proof}
For \eqref{eq:secondorder} we use the symmetries of the mesh and the averaging procedure. In other words,
if $w \in \polP_1$, then $w(z_{k,l}) = W_{k,l}$ and $\Pi_h w = w$.
This, together with an argument \emph{\`a la} Bramble-Hilbert, implies the result.

Estimate \eqref{eq:Lqapprox} follows from \eqref{eq:Lpstab} and \eqref{eq:BVoptimalapprox} 
using the well-known function space interpolation inequality \cite{MR2328004}
\[
  \| w \|_{L^q} \leq \| w \|_{L^1}^{1-s} \| w \|_{L^p}^s,
  \quad \frac1q = \frac{1-s}1 + \frac{s}{p}.
\]

We only sketch the proof of the first statement, as the arguments are similar to those used to obtain
Theorem~\ref{thm:tvd}. We write
\[
  \| w - \Pi_h w \|_{L^1} = \sum_{\kappa,\ell=0}^{N_1-1,N_2-1} \int_{T_{\kappa,\ell}} | w - \Pi_h w |,
\]
and estimate the local differences as
\[
  | w - \Pi_h w | \leq | w - W_{k_0,l_0}| + |W_{k_0,l_0} - \Pi_h w |,
\]
where $(k_0,l_0)$ is such that $z_{k_0,l_0}$ is a vertex of $T_{\kappa,\ell}$.
Provided that $w \in \calC^1(\Omega)$ the bound on the first term is immediate, for if $z\in T_{\kappa,\ell}$,
\begin{align*}
 | w(z) - W_{k_0,l_0}| &= 
 \left | \int \left( w(z) - w(z_{k_0,l_0} + \zeta ) \right)\psi(\zeta) \diff \zeta \right| \\
 &\leq c h \int_Q \int_0^1 \left| \GRAD w (sz + (1-s)(z_{k_0,l_0}+\zeta)) \right|
 \psi(\zeta) \diff s \diff \zeta,
\end{align*}
so that
\[
  \sum_{\kappa,\ell=0}^{N_1-1,N_2-1} \int_{T_{\kappa,\ell}} | w - W_{k_0,l_0} |
  \leq c h \int_\Omega |\GRAD w |,
\]
because $\| \psi \|_{L^\infty} = (h_1h_2)^{-1} = |T_{\kappa,\ell}|^{-1}$.
For the second term we recall that the basis functions form a partition of unity, whence
\begin{align*}
  |W_{k_0,l_0} - \Pi_h w(z)| &= 
  \left|W_{k_0,l_0} 
      \sum_{k=\kappa,l=\ell}^{\kappa+1,\ell+1} \Lambda_{k,l}(z) 
    - \sum_{k=\kappa,l=\ell}^{\kappa+1,\ell+1} W_{k,l} \Lambda_{k,l}(z) \right|\\
  &\leq \sum_{k=\kappa,l=\ell}^{\kappa+1,\ell+1} \Lambda_{k,l}(z) \left| W_{k,l} - W_{k_0,l_0} \right|,
\end{align*}
Using, again, the differentiability properties of $w$,
\[
  | W_{k,l} - W_{k_0,l_0}| \leq h \int_Q \int_0^1 
  \left| \GRAD w \left( sz_{k,l} + (1-s)(z_{k_0,l_0}+\zeta) \right) \right|
  \diff s \psi(\zeta)\diff \zeta,
\]
which, employing the facts that $|\Lambda_{k,l}(z)|\leq 1$ and 
$\| \psi \|_{L^\infty} = (h_1h_2)^{-1}$, implies
\[
  \sum_{\kappa,\ell=0}^{N_1-1,N_2-1} \int_{T_{\kappa,\ell}} | \Pi_h w - W_{k_0,l_0} | 
  \leq c h \int_\Omega |\GRAD w |.
\]
The proof concludes via a density argument with $\calC^1$ functions.
\end{proof}

\section{TV diminishing interpolant: non-periodic domains}
\label{sec:nonperiodic}

Let us describe two possible ways of modifying our construction in the case $\Omega = (0,1)^2$.
Both constructions will satisfy \eqref{eq:tvd-def} and possess optimal approximation properties
in the space $\BV$, so that all the results presented in Section~\ref{sec:Applications} still
hold in this case.
Property \eqref{eq:secondorder} however, will no longer be valid in general.

\subsection{Interpolation based on homotetic transformations of the domain}
\label{sub:homotety}
Set $\epsilon = \tfrac12\max\{h_1,h_2\}$, we define
\[
  \Omega_\epsilon = \left(-\epsilon, 1+\epsilon \right)^2
\]
and notice that $\Omega \Subset \Omega_\epsilon$. For a function $w\in\BV$ we define
\[
  w_\epsilon(y) = \frac1{1+2\epsilon} w \left( \frac{y+\epsilon}{1+2\epsilon} \right),
  \quad y \in \Omega_\epsilon.
\]
Notice that, employing the change of variables $x=(y+\epsilon)/(1+2\epsilon)$ valid for functions of
bounded variation \cite[Theorem 3.16]{MR1857292}, we obtain
\begin{align*}
  \TV_{\Omega_\epsilon}( w_\epsilon )
  &= \frac1{(1+2\epsilon)^2}
  \int_{\Omega_\epsilon} \left|\GRAD w \left( \frac{y+\epsilon}{1+2\epsilon} \right) \right| \diff y
  = \frac{(1+2\epsilon)^2}{(1+2\epsilon)^2} \int_\Omega |\GRAD w(x) | \diff x \\
  &= \TV_\Omega (w),
\end{align*}
where by $\TV_A(w)$ we have denoted the total variation of the function $w$ over the set $A$.
Clearly $w_\epsilon |_\Omega \in \BV$.

We define $\calI_h w = \Pi_h w_\epsilon$, where the operator $\Pi_h$ is defined as in
Section~\ref{sec:interpolant}. Notice that the averaging procedure necessary for the definition of
$\Pi_h w_\epsilon$ makes sense, since the function $w_\epsilon$ is defined in $\Omega_\epsilon$.
The results of Section~\ref{sec:interpolant} then imply
\[
  \TV_\Omega(\calI_h w ) = \TV_\Omega(\Pi_h w_\epsilon) \leq \TV_{\Omega_\epsilon}(w_\epsilon)
  = \TV_\Omega(w).
\]
The approximation properties \eqref{eq:BVoptimalapprox} and \eqref{eq:Lqapprox}
of $\calI_h$ can be obtained similarly.

\begin{rem}[Second order approximation]
The approximation property \eqref{eq:secondorder} cannot hold for $\calI_h$.
In fact if $w\equiv c$, then $\calI_h w = c/(1+2\epsilon)$.
This also shows that the
operator does not preserve constants, which might be an undesirable feature.
\end{rem}

\subsection{Interpolation based on rescaled local averages}
\label{sub:localaverage}
The construction of \S\ref{sub:homotety}, although simple, does not preserve constants. For this reason
we consider a slightly more complicated procedure. It is a modification of the construction
of Section~\ref{sec:interpolant} to take into account the effects of the boundary.

The operator $\calC_h : \Lun \ni w \mapsto \calC_h w \in X_h$ is then such that
if $z_{k,l}$ is an internal node, the value of the interpolant is as in \eqref{eq:defofPi}.
On the other hand, if $z_{k,l}$ is a boundary node, we define
\[
  Q_\Omega(k,l) = (z_{k,l}+Q)\cap \Omega,
\]
and
\begin{equation}
  W_{k,l} = \frac1{|Q_\Omega(k,l)|} \int_{Q_\Omega(k,l)} w(y) \diff y.
\label{eq:defofWklbdry}
\end{equation}
Observe that, $|Q|/|Q_\Omega(k,l)|=4$ for vertices of $\Omega$ and $|Q|/|Q_\Omega(k,l)|=2$
for the rest of boundary nodes.
This rescaling allows us to take into account the lack of symmetry on the interaction with neighboring
nodes.

\begin{thm}[TV diminishing: non-periodic case]
\label{thm:tvdnonperiodic}
The operator $\calC_h : \Lun \rightarrow X_h$ defined by $\calC_h w (z_{k,l})=W_{k,l}$ with
$W_{k,l}$ as in \eqref{eq:defofWkl} for $k=\overline{1,N_1-1}$, $l=\overline{1,N_2-1}$ and by 
\eqref{eq:defofWklbdry} otherwise, satisfies \eqref{eq:tvd-def}.
\end{thm}
\begin{proof}
We split $Q$ into four congruent rectangles
\begin{align*}
  Q^{++} &= \{ z \in Q : z^1>0, \ z^2 > 0 \},  &Q^{+-} &= \{ z \in Q : z^1>0, \ z^2 < 0 \}, \\
  Q^{-+} &= \{ z \in Q : z^1<0, \ z^2 > 0 \},  &Q^{--} &= \{ z \in Q : z^1<0, \ z^2 < 0 \},
\end{align*}
and introduce
\[
  R^{-+} = Q^{-+} \cup Q^{++}, \qquad R^{+-} = Q^{+-} \cup Q^{++}.
\]

It is necessary to compute the contribution of the cells that intersect the boundary.
To illustrate the procedure we consider $T_{0,0}$ and recall \eqref{eq:sharp} to write:
\begin{align*}
    \frac2{h_2} \int_{T_{0,0}} |\partial_1 \calC_h w | &\leq
        \left| \frac1{|R^{-+}|} \int_{R^{-+}} w(z_{1,0} + \zeta) \diff \zeta
      - \frac1{|Q^{++}|} \int_{Q^{++}} w(z_{0,0} + \zeta) \diff \zeta \right|
      \\
      &+  \left|\frac1{|Q|}\int_Q w(z_{1,1} + \zeta) \diff \zeta
      - \frac1{|R^{+-}|} \int_{R^{+-}} w(z_{0,1}+\zeta) \diff \zeta \right|
    .
\end{align*}
The first two terms on the right hand side can be estimated as
\begin{multline*}
  \left| \frac2{h_1 h_2} \int_{R^{-+}} w(z_{1,0} + \zeta) \diff \zeta
  - \frac4{h_1 h_2} \int_{Q^{++}} w(z_{0,0} + \zeta) \diff \zeta \right| \\
  \leq \frac2{h_2}
  \int_{Q^{++}} \int_0^1 |\partial_1 w( s z_{0,1} + (1-s)z_{0,0} + \zeta )| \diff \zeta \diff s \\
  + \frac2{h_2} \int_{Q^{-+}} \int_0^1 
    \left| \partial_1 w \left( sz_{1,0}+(1-s)(z_{0,0} + \tfrac{h_1}2 \be_1 + \zeta
    ) \right) \right| \diff \zeta \diff s.
\end{multline*}
The summation procedure is now as before. For this it suffices to notice
that $|Q^{++}|=|Q^{+-}|=|Q^{-+}|=|Q^{--}|=\tfrac14 |Q|$.
\end{proof}

The operator $\calC_h$ is $L^p$-stable as in \eqref{eq:Lpstab}.
Its approximation properties are as in \eqref{eq:BVoptimalapprox} and \eqref{eq:Lqapprox}. For
brevity we skip the proofs, as they repeat the arguments we have already presented.
Property \eqref{eq:secondorder} however, will not be valid in general.

\begin{rem}[Second order approximation]
Denote
\[
  \calN_\vare = \left\{ x \in \Omega :\ \dist(x, \partial\Omega) < \vare \right\}.
\]
The proof of Theorem~\ref{thm:tvdnonperiodic} shows that if
$w \in W^1_\infty(\calN_\vare)$ and $\tfrac32 h \leq \vare$, then
\[
  \| w - \calC_h w \|_{L^\infty(\calN_h)} \leq c h |w|_{W^1_\infty(\calN_{3h/2})}.
\]
Therefore, since $|\calN_h| \leq c h$, we deduce
\[
  \| w - \calC_h w \|_{L^1} \leq c h^2.
\]
\end{rem}

\section{Applications}
\label{sec:Applications}
In this Section we present two applications of the operators constructed in Sections \ref{sec:interpolant}
and \ref{sec:nonperiodic}.
The fact that the operators satisfy \eqref{eq:tvd-def} and \eqref{eq:BVoptimalapprox}
allows us to not only improve on existing results
but also to have simpler arguments.

\subsection{Total Variation Minimization}
\label{sec:minmize}
Let $\alpha>0$ and $f \in L^\infty(\Omega)$. Set
\begin{equation}
  E(w) = \TV(w) + \frac\alpha2 \| w - f \|_{L^2}^2,
\label{eq:tvmin}
\end{equation}
and consider the problem of finding $\ve \in \BV$ that minimizes this functional. This problem arises,
for instance, in connection with image processing \cite{Rudin1992259}. 
In \cite{BartelsTV} S.~Bartels presents a first order finite element method for the approximation
of minimizers of \eqref{eq:tvmin} as well as an iterative scheme for their computation, thereby
extending prior work of A.~Chambolle and T.~Pock \cite{MR2782122} and
B.~Lucier and J.~Wang \cite{MR2792398}.
The main convergence result of this work reads as follows: Let
\begin{equation}
  \ve = \argmin_{w \in \BV} E(w),
\label{eq:defofu}
\end{equation}
and
\begin{equation}
  v_h = \argmin_{w_h \in X_h } E(w_h).
\label{eq:defofuh}
\end{equation}
If $\ve \in B^\beta_\infty(\Ldeux)$, with $\beta \in (0,1]$, then
\begin{equation}
  \| \ve - v_h \|_{L^2} \leq c h^{\frac\beta{2(1+\beta)} }.
\label{eq:knowntvmin}
\end{equation}
Let us, with the help our TV diminishing operator, improve on this result,
which is at best $\calO(h^{1/4})$.

\begin{thm}[Total variation minimization]
\label{thm:tvmin}
Let $\ve$ be as in \eqref{eq:defofu} and $v_h$ as in \eqref{eq:defofuh}. Assume that
$\ve \in \BV \cap L^\infty(\Omega)$, then
\[
  \| \ve - v_h \|_{L^2} \leq \sqrt{2} h^{1/2} \left( \| \ve \|_{L^\infty} + \| f \|_{L^\infty} \right)^{1/2}
  \TV(\ve)^{1/2}.
\]
\end{thm}
\begin{proof}
Owing to the strict convexity of $E$ and the fact that $v_h$ is a minimizer over $X_h$, we have
\begin{align*}
  \frac\alpha2 \| \ve - v_h \|_{L^2}^2 &\leq E(v_h) - E(\ve)
    \leq E(\Pi_h \ve) - E(\ve) \\
    &\leq \TV(\Pi_h \ve) - \TV(\ve) 
    + \frac\alpha2 \left( \| \Pi_h \ve - f \|_{L^2}^2 - \| \ve - f \|_{L^2}^2 \right).
\end{align*}
Since $\Pi_h$ satisfies \eqref{eq:tvd-def} and \eqref{eq:Lpstab},
\begin{align*}
  \| \ve - v_h \|_{L^2}^2 &\leq
  \scl \Pi_h \ve - \ve , \Pi_h \ve + \ve - 2f \scr
  \leq 2\| \Pi_h \ve- \ve \|_{L^1} \left( \| \ve \|_{L^\infty}+  \| f \|_{L^\infty} \right) \\
  &\leq 2 h \left( \| \ve \|_{L^\infty}+  \| f \|_{L^\infty} \right)\TV(\ve),
\end{align*}
where, to obtain the last inequality, we used Theorem~\ref{thm:approx}.
\end{proof}

\begin{rem}[Reduced regularity]
The conclusion of Theorem~\ref{thm:tvmin} is indeed an improvement over \eqref{eq:knowntvmin}.
It is not difficult to show that
$\BV \cap L^\infty(\Omega) \hookrightarrow B^{1/2}_\infty(\Ldeux)$;
see \cite[Lemma 38.1]{MR2328004} for a proof and, in some sense, the converse statement. In this
case \eqref{eq:knowntvmin} yields a $\calO(h^{1/6})$ error estimate. Our results yields a
rate which, given the smoothness of $\ve$, is optimal.
\end{rem}

\subsection{TV flows}
The equation
\begin{equation}
  \ue_t = \DIV \left( \dfrac{ \GRAD \ue  }{|\GRAD \ue |} \right),
\label{eq:tvflow}
\end{equation}
supplemented with suitable initial and boundary conditions is known as the \emph{TV flow}, since it can be
interpreted as the $L^2$-subgradient flow of the functional $\TV$. We refer the reader to
\cite{MR2033382,MR0348562,MR2746654,MR1865089}
for the analysis of this problem and to \cite{TVFlow,MR1994316,MR2194526} for its
discretization. Concerning applications this problem,
supplemented with a smooth strictly convex $p$-Laplacian term,
arises, \eg, in the modeling of 
grain boundary motion \cite{Kobayashi2000141}; facet formation and evolution \cite{PhysRevB.78.235401};
and electromigration \cite{Fu1997259}.

The precise mathematical meaning of \eqref{eq:tvflow} is as a subgradient flow. In other words,
given $\ue_0 \in \Ldeux$, we seek a function $\ue: [0,T] \rightarrow \Ldeux\cap\BV$ such that
\begin{equation}
  \begin{dcases}
    \scl \ue_t, \ue - w \scr + \TV(\ue) - \TV(w) \leq 0, & \forall w \in \Ldeux\cap\BV, \\
    \ue|_{t=0} = \ue_0.
  \end{dcases}
\label{eq:tvflowvarineq}
\end{equation}

In prior work \cite{TVFlow}, we provided an error analysis for discretizations of
\eqref{eq:tvflowvarineq} as well as a convergent solution scheme for it.
We considered the fully discrete scheme:
Find $\{u_h^k\}_{k=0}^K \subset X_h$ such that
\begin{equation}
  \begin{dcases}
    \scl \frac{ u_h^{k+1} - u_h^k }\dt, u_h^{k+1} - w_h \scr 
    + \TV(u_h^{k+1}) - \TV(w_h) \leq 0, & \forall w_h \in X_h, \\
    u_h^0 = \calI_h \ue_0.
  \end{dcases}
\label{eq:tvflowdiscrete}
\end{equation}
where $\dt >0$ is a time-step, $K = [T/\dt]$ and $\calI_h$ is some interpolation operator.
Our main convergence result stated that if $\ue_0 \in L^\infty(\Omega)\cap\BV$ and
$\TV(\calI_h \ue_0 ) \leq c < \infty$ uniformly in $h$, then
\[
  \| \ue - \widehat u_h \|_{L^\infty(L^2)} \leq   \| \ue_0 - \calI_h u_0 \|_{L^2}
  + c \left( \dt^{1/2} + h^{1/6} \right),
\]
where $\widehat u_h$ denotes the piecewise linear (in time) interpolant of the sequence $\{u_h^k\}$.
This result, as \cite{MR1737503,MR1377244} show, is optimal with respect to time, given the
regularity of the initial data.

Let us, with the help of the TV diminishing operator of Sections \ref{sec:interpolant}
and \ref{sec:nonperiodic} improve the convergence rate in space.

\begin{thm}[Convergence of TV flows]
Let $\ue$ solve \eqref{eq:tvflowvarineq}. Assume that the initial data $\ue_0 \in L^\infty(\Omega)\cap\BV$.
Let $\{ u_h^k \}_{k=0}^K \subset X_h$ solve \eqref{eq:tvflowdiscrete}, with $u_h^0 = \Pi_h \ue_0$.
Then
\begin{equation}
  \| \ue - \widehat u_h \|_{L^\infty(L^2)} \leq c \left( \dt^{1/2} + h^{1/4} \right).
\label{eq:convrageTVflow}
\end{equation}
\end{thm}
\begin{proof}
The error of the time discretization is obtained with standard techniques, \cf \cite{MR1737503,MR1377244}.
For this reason we concentrate on the error of space discretization and compare the solution of
\eqref{eq:tvflowdiscrete} with the sequence $\{u^k\}_{k=0}^K \subset \Ldeux \cap \BV$, solution of
the semidiscrete problem
\[
  \begin{dcases}
    \scl \frac{ u^{k+1} - u^k }\dt, u^{k+1} - w \scr 
    + \TV(u^{k+1}) - \TV(w) \leq 0, & \forall w \in \Ldeux \cap \BV, \\
    u^0 = \ue_0.
  \end{dcases}  
\]
Set $w = u_h^{k+1}$ in this inequality and $w_h = \Pi_h u^{k+1}$ in \eqref{eq:tvflowdiscrete} and add them.
Denoting $e^k = u^k - u_h^k$, we obtain 
\[
  \| e^{k+1} \|_{L^2}^2 + \| e^{k+1} - e^k \|_{L^2}^2 \leq \| e^k \|_{L^2}^2 +
  2\dt \left\| \frac{ u_h^{k+1} - u_h^k }\dt \right\|_{L^2} \| u^{k+1} - \Pi_h u^{k+1} \|_{L^2}
\]
where we used \eqref{eq:tvd-def} to cancel the TV terms.

Setting $w_h = u_h^k$ in \eqref{eq:tvflowdiscrete} is not difficult to show that 
$\dt^{-1} \| u_h^{k+1} - u_h^k \|_{L^2} \leq c \TV(\ue_0)$ uniformly in $h$.

The results of \cite{TVFlow} show that if $\ue_0 \in L^\infty(\Omega)\cap \BV$, then the solution to
the semidiscrete flow $\{u^k\} \subset L^\infty(\Omega)\cap \BV$.
Consequently, it suffices to invoke Theorem~\ref{thm:approx} and add over $k$.
\end{proof}

\begin{rem}[Error in space]
At this stage one might ask whether an optimal rate of convergence $\calO(h^{1/2})$ can
be obtained for \eqref{eq:tvflowdiscrete}. Our methods show that this is only possible under
higher regularity on the solution. To realize this the reader can repeat our analysis for the 
heat equation (where the arguments are simpler) and convince himself that optimality is a consequence of the
higher regularity that is expected from the solution. 
Unfortunately, the functional $\TV$ does not have any smoothing
properties beyond the ones we have used so it not possible to obtain a better rate of convergence
with the present techniques.
\end{rem}

\bibliographystyle{plain}
\bibliography{biblio}

\end{document}